\documentclass[11pt,twoside]{article}
\usepackage[T1]{fontenc}
\usepackage{amsmath}
\usepackage{amsfonts}
\usepackage{amsthm}
\usepackage{graphicx}
\usepackage{color}
\usepackage{framed}
\usepackage{bm}
\usepackage{mathtools}
\usepackage{cite}
\usepackage{pgf}
\usepackage{lmodern}
\usepackage{pifont}
\usepackage{multirow,booktabs}
\usepackage{url}
\usepackage{tikz}

\newcommand{\pd}[2]{\frac{\partial #1}{\partial #2}}
\newcommand{\di}{\textnormal{div}\, }
\newcommand{\dix}{\textnormal{div}_x\, }
\newcommand{\diR}{\textnormal{div}_{R}\, }

\newcommand{\nablax}{\nabla_x\,}
\newcommand{\nablaR}{\nabla_{R}\,}
\newcommand{\DeltaR}{\Delta_{R}\,}

\newcommand{\Rd}{\mathbb{R}^d}

\newcommand{\tM}{\tiny \textsc{m}}
\newcommand{\tP}{\tiny \textsc{p}}

\newcommand{\xko}{\mathbf{x}}
\newcommand{\nko}{\mathbf{n}}
\newcommand{\ucko}{\mathbf{u}}
\newcommand{\vcko}{\mathbf{v}}
\newcommand{\Ccko}{\mathbf{C}}
\newcommand{\Dcko}{\mathbf{D}}
\newcommand{\Tcko}{\mathbf{T}}
\newcommand{\Rko}{\mathbf{R}}
\newcommand{\Icko}{\mathbf{I}}
\newcommand{\Fko}{\mathbf{F}}
\newcommand{\trC}{\textnormal{tr}\, \mathbf{C}}

\newcommand{\intO}[1]{\int_{\Omega} {#1}\ dx\, }

\newcommand{\intoto}[2]{\int_0^{#1} \int_{\Omega} {#2}\ dx \, dt\, }
\newcommand{\intR}[1]{\int_{\mathbb{R}^d} {#1} \ d\Rko \, }
\newcommand{\intOR}[1]{\int_{\Omega \times \Rd} {#1} \    d\Rko\,dx\,}
\newcommand{\intOTOR}[1]{\int_0^T \int_{\Omega\times\Rd}{#1} \ d\Rko\,dx\,dt \, }

\newcommand{\LpM}[1]{L_M^{#1}(\Omega \times \Rd)}
\newcommand{\HkM}[1]{H_M^{#1}(\Omega \times \Rd)}
\newcommand{\LLko}[2]{L^{#1}(0,T;{#2})}
\newcommand{\Lp}[1]{L^{#1}((0,T)\times\Omega)^{d \times d}}

\newcommand{\n}[2]{\left\lVert{#2}\right\rVert_{#1}}
\newcommand{\nLpM}[2]{\|{#2}\|_{L_M^{#1}(\Omega \times \Rd)}}
\newcommand{\nHkM}[2]{\|{#2}\|_{H_M^{#1}(\Omega \times \Rd)}}

\newcommand{\varphih}{\hat{\varphi}}
\newcommand{\us}{\ucko_*}
\newcommand{\psis}{\psi_*}
\newcommand{\psih}{\hat{\psi}}
\newcommand{\psish}{\hat{\psi}_*}
\newcommand{\psishL}{\hat{\psi}_{*,L}}

\newcommand{\trCs}{\trC_*}
\newcommand{\Cs}{\Ccko_*}
\newcommand{\betaL}{\beta^L}

\newcommand{\go}{\gamma_1}
\newcommand{\gt}{\gamma_2}

\renewcommand{\(}{\left(}
\renewcommand{\)}{\right)}

\DeclareMathOperator*{\esssup}{ess\,sup}

\newcommand{\defeq}{\vcentcolon=}
\newcommand{\eqdef}{=\vcentcolon}

\newtheorem{theo}{Theorem}
\newtheorem{defin}[theo]{Definition}

\newtheorem{lem}[theo]{Lemma}

\newtheorem{coro}[theo]{Corollary}
\theoremstyle{remark}

\allowdisplaybreaks

\title{Existence of global weak solutions to the kinetic Peterlin model}
\author{P.~Gwiazda, M.~Luk\' a\v cov\' a-Medvi\v dov\' a, H.~Mizerov\' a, A.~\' Swierczewska-Gwiazda}

\begin{document}

\maketitle

\begin{abstract}
We consider a class of kinetic  models for polymeric fluids motivated by the Peterlin dumbbell theories for dilute polymer solutions with a nonlinear spring law for an infinitely extensible spring. The polymer molecules are suspended in an incompressible viscous Newtonian fluid confined to a bounded domain in two or three space dimensions. The unsteady motion of the solvent is described by the incompressible Navier-Stokes equations with the elastic extra stress tensor appearing as a forcing term in the momentum equation. The elastic stress tensor is defined by the Kramers expression through the probability density function that satisfies the corresponding Fokker-Planck equation. In this case, a coefficient depending on the average length of  polymer molecules appears in the latter equation. Following the recent work of Barrett and S\" uli \cite{BaSu6} we prove the existence of global-in-time weak solutions to the kinetic Peterlin model in two space dimensions. 
\end{abstract}


\section{Introduction}

The Peterlin approximation is a nonlinear model falling into the category of Navier-Stokes-Fokker-Planck type systems. The nonlinearity of the model corresponds to the nonlinear spring potential for infinitely extensible molecular chains appearing in the Fokker-Planck equation. Among the nonlinear dumbbell models the most commonly studied one is the FENE model - finitely extensible nonlinear elastic model. Its advantage consists in a particular form of the spring potential, which forces that the system is considered in a bounded domain. Thus, even though in case of such a nonlinearity the macroscopic closure is not possible, but the methods developed in a series of papers, cf.~\cite{BaSu, BaSu4, BaSu5} allowed for showing existence of global-in-time weak solutions. The case of spring potential in the Peterlin model does not provide finite extensibility of polymeric chains, thus the problem of unbounded domain (and integration by parts) has to be faced. However, the idea of averaging t
 he coefficients (with respect to $\Rko$ - the vector corresponding to the length and orientation of polymers) gave that they depend on the macroscopic quantity only, namely the trace of the conformation tensor $\trC=\langle|\Rko|^2\rangle$, which is the average length of polymer molecules suspended in the solvent. This property apparently allows to prove a rigorous macroscopic closure of a kinetic equation and to use the results on existence and regularity of macroscopic quantities. This idea has very recently been used for a linear Hookean dumbbell model 
 \cite{BaSu6} by Barrett and S\" uli, who   showed the existence of large-data global-in-time 
weak solutions. The Hookean model arises as a microscopic-macroscopic bead-spring model from the kinetic theory of dilute solutions of polymeric liquids with noninteracting polymer chains. The authors have also rigorously showed that the well-known Oldroyd-B model is a macroscopic closure of the Hookean dumbbell model in two space dimensions.  It is worth to mention here that an attempt of mimicking the approach used for FENE models to linear Hookean case failed. Barrett and S\"uli in~\cite{BaSu2} firstly covered just the case of Hookean-type models, meaning by that a slight modification of the spring potential to provide the uniform integrability of appropriate terms.  

Motivated by their approach we study the kinetic Peterlin model representing a class of kinetic dumbbell-based models for dilute polymer solutions with a nonlinear spring force law.
For the macroscopic closure of the corresponding kinetic equation, it is necessary to approximate the spring force. We consider the Peterlin approximation  \cite{P} that allows us to derive the so-called Peterlin viscoelastic model, which has been studied in our recent work \cite{LuMiNeRe,mizerova} and these results will be essentially  used in a current approach. See also \cite{LuMiNoTaI,LuMiNoTaII} for our recent result on error analysis using the Lagrange-Galerkin method.
As a consequence of the approximation of the force law, the Fokker-Planck equation contains  additional coefficients, which depend on the trace of the (macroscopic) conformation tensor. 
As mentioned in \cite{renardy_wang}, the Peterlin model can be therefore seen as the generalization of the upper-convected Maxwell model, in which the relaxation time and viscosity depend on a ``structure parameter''  $\trC.$

In Section~\ref{sec:model} of the present paper we introduce the kinetic Peterlin model and its formal macroscopic closure, the  so-called Peterlin viscoelastic model. In the next section we recall our recent results on uniqueness of regular weak solutions to the proposed macroscopic model. In Section~\ref{sec:FP}, recalling the idea of Barrett and S\" uli \cite{BaSu6}, we show the existence of global-in-time weak solutions to the Fokker-Planck equation for some given fluid velocity $\us$ and conformation tensor $\Cs$. Let us mention that the difference between the kinetic Peterlin model analysed in the present paper and the Hookean model studied in \cite{BaSu6} is the dependence on the structure parameter $\trC$ appearing in the Fokker-Planck equation due to the Peterlin approximation of the nonlinear spring force law. Finally, in Section~\ref{sec:ex_kinetic} we show the existence of global-in-time weak solutions to the kinetic Peterlin model in two space dimensions. We combine the 
 result on uniqueness of  solutions to the macroscopic model with the results presented in Section~\ref{sec:FP}.


\section{The kinetic Peterlin model}\label{sec:model}

In the present paper we study the existence of global weak solutions to a kinetic dumbbell-based model for dilute polymer solutions. The polymer molecules are suspended in an incompressible viscous Newtonian fluid confined to an open bounded domain $\Omega \subset \mathbb R^d,$ $d=2,3.$   The incompressible Navier-Stokes equations equipped with the no-slip boundary condition for the velocity are used to describe the unsteady motion of the solvent. 

Let $T > 0$ be given, find $\ucko:  [0,T]\times\bar{\Omega} \rightarrow
\mathbb R^d$ and $p: [0,T]\times \bar{\Omega} \rightarrow
\mathbb R$ such that 
\begin{subequations}\label{NS}
\begin{align}
\pd{\ucko}{t} + (\ucko\cdot\nabla_x)\ucko &= \nu\Delta_x\ucko + \dix\Tcko - \nabla_x p  & & \mbox{ in }  (0,T)\times \Omega, \label{ucko_eq} \\
\dix \ucko &= 0 & & \mbox{ in } (0,T)  \times  \Omega, \label{div_u}\\
\ucko&=\mathbf{0}   && \mbox{ on }  (0,T)\times\partial\Omega, \label{bc_ucko_NS}\\
\ucko(0)&=\ucko_{0}   && \mbox{ in }  \Omega. \label{ic_ucko_NS}
\end{align}
The elastic extra stress tensor $\Tcko: [0,T]\times\bar{\Omega} \rightarrow
\mathbb R^{d\times d},$ arising due to the random movement of polymers in the solvent, appears as the forcing term in equation \eqref{ucko_eq}, and depends  on the probability density function $\psi.$ It is defined  by the Kramers expression
 \begin{align}
\Tcko(\psi)&= n\gamma_3(\trC(\psi))\Ccko(\psi) -\Icko, \label{Tcko_eq}
\end{align}
\end{subequations}
where $n$ denotes the number density of polymer molecules,  i.e., the number of polymer molecules per unit volume. Let us note that the above equations \eqref{NS} are written in their non-dimensional form; $U_0,$ $L_0$ denote in what follows the characteristic flow speed and the characteristic length-scale of the flow, respectively; viscosity $\nu >0$ is the reciprocal of the Reynolds number.

The polymers are modelled as two beads connected by a spring and are assumed not to interact with each other. The spring connecting the beads exerts a spring force $\Fko(\Rko)$ with $\Rko$ being the vector connecting the beads.  We consider the spring force  to be nonlinear, i.e.  $\Fko(\Rko)=\go(|\Rko|^2)\Rko.$ On each of the beads there is a balance between the spring force, a friction force exerted by the surrounding fluid and a stochastic force due to Brownian motion. 
Let $\zeta >0$ be a friction coefficient, $k\tau$ be the magnitude of stochastic forces with $k$ being the Boltzmann constant and $\tau$ being the absolute temperature. Then the probability density  $\psi: [0,T] \times \Omega \times \mathbb R^d \rightarrow \mathbb R^+_0$  satisfies the following Fokker-Planck equation
\begin{align}\label{psi_eq}
\pd{\psi}{t} + (\ucko\cdot\nabla_x)\psi +\diR\left[\nabla_x\ucko \cdot \Rko \psi\right]&= \frac{2k\tau}{\zeta}\gt(\langle |\Rko|^2\rangle)\DeltaR \psi + \frac{2}{\zeta}\diR\left[ \Fko(\Rko)\psi \right]  +\frac{k\tau}{2\zeta}\Delta_x\psi  
\end{align}
with the center-of-mass diffusion coefficient $(k\tau)/(2\zeta)>0.$ The constitutive functions $\go,$ $\gt$ and $\gamma_{3}$ defined on $\mathbb R$ are from now on assumed to be continuous and positive-valued.
 If they are constant, then we obtain the Hookean dumbbell model whose closure is the well-known Oldroyd-B model. In order to derive an analogous closed system of equations for the conformation tensor we employ the Peterlin approximation of the spring force, which replaces the length of the spring $|\Rko|^2$ in the spring constant $\go$ by the average length of the spring $\langle |\Rko|^2\rangle.$ The force law thus reads $\Fko(\Rko)=\go(\langle |\Rko|^2 \rangle)\Rko=\go(\trC(\psi))\Rko.$ We note that for the trace of the macroscopic conformation tensor $\Ccko(\psi):=\langle \Rko \otimes \Rko\rangle$ it holds that  $\trC(\psi)=\langle |\Rko |^2 \rangle.$ Here $\otimes$ denotes the dyadic product and 
\begin{align*}
\langle f \rangle \defeq \intR{f(\Rko) \psi(t,\xko,\Rko)}  .
\end{align*}
For more details on deriving equation \eqref{psi_eq} we refer the reader to, e.g., \cite{Re1,BaSu6,degond_liu,shieber,P} and the references therein.

\begin{subequations}\label{FK_in_KP}
\begin{defin}(normalized Maxwellian) \label{def:Maxwellian} \\
We define the equilibrium distribution of the probability density function by
\begin{align}\label{Maxwellian}
M:=b\exp\left\{-\frac{|\Rko|^2}{2a}\right\} \quad \mbox{ with } \quad a:=\frac{k\tau\,\gamma_{2,\tM}}{\gamma_{1,\tM}},\ b:=(2\pi a)^{-d/2}.
\end{align}  
Here $\gamma_{i,\tM}:=\gamma_{i}(\trC_{\tM}) >0,$ $i=1,2,$ denote the values  of the functions $\go$ and $\gt$ at the equilibrium. We note that $\trC_{\tM}:=\trC(M)=d.$ 
\end{defin}

The next lemma provides the non-dimensional form of the Fokker-Planck equation \eqref{psi_eq} rewritten using the Maxwellian $M$ defined above.

\begin{lem} \ \label{lem:nondim_FP}\\
Let the functions $\go,$ $\gt$ be such that  the identity
\begin{align} \label{m2_Mcond}
\frac{\gamma_{1,\tM}}{\gamma_{2,\tM}}= \frac{\go(\trC)}{\gt(\trC)}=k\tau \eqdef \gamma_{\tM}
\end{align}
is satisfied for a.e. $(t,x) \in (0,T)\times \Omega.$
Then the Fokker-Planck equation \eqref{psi_eq} can be rewritten in its non-dimensional form as 
\begin{align}\label{FP_M}
\pd{\psi}{t} + (\ucko\cdot\nabla_x)\psi +\diR\left[\nabla_x\ucko \cdot \Rko \psi\right]&=  \Gamma(\trC)\,\nablaR \cdot \(M \, \nablaR\(\frac{\psi}{M}\)\)+\varepsilon\Delta_x\psi ,
\end{align}
where the coefficient $\Gamma(\trC)> 0$  and the center-of-mass diffusion coefficient $\varepsilon >0$ are given by 
\begin{align}\label{Gamma_vareps}
\Gamma(\trC):=\frac{\gt(\trC)}{2\lambda} , \quad \varepsilon:=\(\frac{l_0}{L_0}\)^2\frac{1}{8\lambda},
\end{align}
respectively. The coefficient $\lambda:=(\zeta/4\gamma_{\tM})(U_0/L_0),$ usually called the Deborah number, characterises the elastic relaxation property of the fluid, 
and  $l_0:=\sqrt{\trC_{\tM}/d}$ denotes the characteristic miscroscopic length-scale.
\end{lem}
\end{subequations}
  
\begin{proof}
Let us consider the non-dimensional variables denoted by $\sim,$ e.g., $\tilde{\psi}:=\psi/d_0,$ where $d_0$ is the characteristic probability density. We insert these variables  into \eqref{psi_eq} and multiply the resulting equation by $\frac{L_0}{U_0 d_0}.$ On noting $T_0=U_0/L_0,$ equation \eqref{psi_eq} becomes
\begin{equation}\label{psi_eq_nondim}
\begin{aligned}
\pd{\tilde{\psi}}{t} + (\tilde{\ucko}\cdot\nabla_{\tilde{x}})\tilde{\psi} +\di_{\tilde{R}}\left[\nabla_{\tilde{x}}\tilde{\ucko} \cdot \tilde{\Rko} \tilde{\psi}\right]&=\frac{2k\tau\,\gt(\trC)}{\zeta}\frac{L_0}{U_0(l_0)^2}\Delta_{\tilde{R}} \tilde{\psi} + \frac{2}{\zeta}\frac{L_0}{U_0}\di_{\tilde{R}}\left[\go(\trC)\tilde{\Rko}\tilde{\psi}\right]+& \\
& +\frac{k\tau}{2\zeta}\frac{1}{L_0 U_0} \Delta_{\tilde{x}}\tilde{\psi}.
\end{aligned}
\end{equation}
The direct calculation, taking into account \eqref{Maxwellian}, yields 
\begin{align*}
 \nabla_{\tilde{R}} \cdot \(\tilde{M} \, \nabla_{\tilde{R}}\(\frac{\tilde{\psi}}{\tilde{M}}\)\)=
\Delta_{\tilde{R}} \tilde{\psi} + \frac{1}{a}\di_{\tilde{R}}\left[\tilde{\Rko}\tilde{\psi}\right]
. 
\end{align*}
 We note that $a=l_0=1.$ 
Thus, it holds that 
\begin{align*}
\frac{2k\tau\,\gt(\trC)}{\zeta}\frac{L_0}{U_0(l_0)^2}\Delta_{\tilde{R}} \tilde{\psi} + \frac{2}{\zeta}\frac{L_0}{U_0}\di_{\tilde{R}}\left[\go(\trC)\tilde{\Rko}\tilde{\psi}\right]
= \frac{2k\tau\,\gt(\trC)}{\zeta}\frac{L_0}{U_0(l_0)^2}\nabla_{\tilde{R}} \cdot \(\tilde{M} \, \nabla_{\tilde{R}}\(\frac{\tilde{\psi}}{\tilde{M}}\)\). 
\end{align*}
By the definition of $\lambda$ and $\gamma_{\tM}$  it holds that  
\begin{align*}
\frac{2k\tau\,\gt(\trC)}{\zeta}\frac{L_0}{U_0(l_0)^2}=\frac{\gt(\trC)}{2\lambda}=
\Gamma(\trC), \quad  \varepsilon=\frac{k\tau}{2\zeta}\frac{1}{L_0 U_0}.
\end{align*}
Omitting the  $\sim$ -notation we get equation \eqref{FP_M}.
\end{proof}
Let us point out that the $\sim$ -notation of  the non-dimensional variables has been used only in the proof of Lemma~\ref{lem:nondim_FP}. In what follows  all the equations are non-dimensional.

\begin{subequations}\label{bc_ic_FP_M}
Finally, we impose the following decay/boundary and initial conditions on $\psi:$
\begin{align}
\left| M \bigg( \Gamma(\trC)\nablaR\left(\frac{\psi}{M}\right)-(\nablax\ucko)\Rko\frac{\psi}{M}\bigg)\right| \rightarrow 0 \quad \mbox{as} \quad |\Rko| \mapsto \infty && \quad \mbox{ on } (0,T] \times\Omega, \\
\varepsilon\pd{\psi}{\nko}=0  && \quad \mbox{ on }  (0,T) \times\partial\Omega \times \Rd, \\
\psi(0)=\psi_0 &&\quad\mbox{ on }  \Omega \times \Rd ,
\end{align}
where $\nko$ is the unit outward normal vector on $\partial\Omega$ and $\psi_0$ is a given non-negative function defined on $\Omega \times \Rd$ with $ \displaystyle \intR{\psi_0(\xko,\Rko)}=1$ for a.e. $x \in \Omega.$
\end{subequations}

\begin{defin}
Throughout the paper we refer to the system of equations and conditions \eqref{NS}, \eqref{FK_in_KP}, \eqref{bc_ic_FP_M} as the kinetic Peterlin  model \textnormal{(KP)}.
\end{defin}

In order to obtain a formal macroscopic closure of the above introduced kinetic model we multiply the non-dimensional form of the Fokker-Planck equation \eqref{psi_eq_nondim} by $\Rko\otimes \Rko$ and integrate by parts over $\Rd$ to get that the conformation tensor $\Ccko:[0,T]\times \Omega \rightarrow \mathbb R^{d \times d}$ satisfies the following euation
\begin{subequations}\label{macro_C}
\begin{align}
\pd{\Ccko}{t} + (\ucko\cdot\nabla)\Ccko - (\nabla\ucko)\Ccko -\Ccko(\nabla\ucko)^T &= \frac{\gt(\trC)}{\lambda} \Icko - \frac{\go(\trC)}{\lambda\gamma_{\tM}}\Ccko +\varepsilon\Delta\Ccko & & \mbox{ in } (0,T)  \times  \Omega \label{C_eq}
\end{align}
subject to the boundary and initial conditions
\begin{align}\label{bc_ic_macroC}
&\varepsilon\pd{\Ccko}{\nko}=0  \quad \mbox{ on }  (0,T) \times\partial\Omega, &
\Ccko(0)=\Ccko_0 \mbox{ in }  \Omega.
\end{align}
\end{subequations}

\begin{defin}
Throughout the paper we refer to the system of equations and conditions \eqref{NS},  \eqref{macro_C} as the (macroscopic) Peterlin model (MP).
\end{defin}

\subsection{Notation and preliminaries}

Let $\Omega \subset \mathbb R^d,$ $d=2,3,$ be a bounded domain with smooth boundary $\partial\Omega.$ We define the following functional spaces
\begin{align*}
 V &\defeq \{ \vcko \in H_0^1(\Omega)^d : \ \dix\vcko=0 \}, \quad  H \defeq\{ \vcko \in L^2(\Omega)^d : \ \dix\vcko=0 , \vcko\cdot{\bf n}=0{\rm\ on\ }{\partial\Omega} \}, &
\end{align*}
where the divergence is understood in the sense of distributions. We shall use the notation \begin{align*}
\psih\defeq \frac{\psi}{M}
\end{align*}
and the Maxwellian-weighted $L^p$ space over $\Omega \times  \mathbb R^d$ denoted by $\LpM{p},$ $p\in [1,\infty),$ equipped with the norm 
\begin{align*}
\nLpM{p}{\varphih}\defeq \left(\intOR{M|\varphih|^p}\right)^{1/p}.
\end{align*} 
 Analogously, we define the space $ \hat{X}\equiv \HkM{1} \defeq \{ \varphih \in L^1_{loc}(\Omega\times \Rd) : \ \nHkM{1}{\varphih}< \infty \} $ with the norm 
 \begin{align*}
\nHkM{1}{\varphih}\defeq   \left(\intOR{M\bigg[|\varphih|^2+|\nablax\varphih|^2+|\nablaR\varphih|^2\bigg]}\right)^{1/2}.
 \end{align*}
  Finally,  let 
\begin{align*}
\hat{Z}_2 \defeq  \left\{ \varphih \in \LpM{2} : \ \varphih \geq 0  \mbox{ a.e. on } \Omega \times \Rd; \intR{M(\Rko)\varphih(x,\Rko)} \leq 1 \mbox{ for a.e. } x \in \Omega\right\}. 
\end{align*} 

The proof of existence of weak solutions to the Fokker-Planck equation is based on the compactness theorem due to Dubinski\u{\i} \cite{Dub}, that is a generalization of the Lions-Aubin compactness theorem. We refer to \cite{Dub,BaSu6} and the references therein for more details.

\begin{theo}(Dubinski\u{\i})\label{theo:dubinski}\\
Suppose that $\mathcal A_0$ and  $\mathcal A_1$ are Banach spaces, $\mathcal A_0 \hookrightarrow \mathcal A_1,$ and $\mathcal M$ is a semi-normed subset of  $\mathcal A_0$ with the compact embedding $\mathcal M \hookrightarrow \mathcal A_0.$ Then, for $\alpha_i >1,$ $i=0,1,$ the embedding
\begin{align*}
\left\{ \eta \in \LLko{\alpha_0}{\mathcal M}\, : \, \pd{\eta}{t} \in \LLko{\alpha_1}{\mathcal A_1}  \right\}  \hookrightarrow  \LLko{\alpha_0}{\mathcal A_0}
\end{align*}
is compact.
\end{theo}

\section{Uniqueness result for the macroscopic Peterlin model}\label{sec:uniq_macro}

One of the crucial parts of the proof of existence of global weak solutions to the kinetic Peterlin model, presented in Section~\ref{sec:ex_kinetic} below, is uniqueness of regular  weak solutions to (MP). In this section we list the available results.

\begin{subequations}\label{weak_MP}
 The couple $(\ucko,\Ccko)$ with
\begin{align}
\ucko \in \LLko{\infty}{H} \cap \LLko{2}{V}, \quad 
 \Ccko \in \LLko{\infty}{L^2(\Omega)^{d\times d}} \cap \LLko{2}{H^1(\Omega)^{d\times d}}
\end{align}
is called a weak solution to the Peterlin model  (MP) if it satisfies, for any $(\vcko,\Dcko) \in V \times H^1(\Omega)^{d\times d},$
\begin{align}
\intO{\pd{\ucko}{t}\cdot\vcko} &+ \intO{\(\ucko\cdot\nabla\)\ucko\cdot\vcko} + \nu\intO{\nabla\ucko:\nabla\vcko} +\intO{\gamma_3(\trC)\Ccko:\nabla\vcko} =0& \label{weak_MP_ucko}\\
\intO{\pd{\Ccko}{t}:\Dcko} &+ \intO{\(\ucko\cdot\nabla\)\Ccko:\Dcko} -2\intO{(\nabla\ucko)\Ccko:\Dcko}+\varepsilon\intO{\nabla\Ccko:\nabla\Dcko}=&\nonumber \\
&=\frac{1}{\lambda}\intO{\gt(\trC)\Icko:\Dcko}-  \frac{1}{\lambda\gamma_{\tM}}\intO{\go(\trC)\Ccko:\Dcko}, &\label{weak_MP_Ccko}
\end{align}
 and if  $(\ucko(0),\Ccko(0))=(\ucko_0,\Ccko_0),$ for a given initial data $(\ucko_0,\Ccko_0)\in H \times L^2(\Omega)^{d\times d}.$
\end{subequations}

\vspace*{0.5cm}
\textbf{Assumptions on the constitutive functions}

Let us assume that $\go,$ $\gt$ and $\gamma_3$ are smooth positive functions defined on $\mathbb R$ and $\gamma_3$ is moreover   non-decreasing. Further, we suppose that for some positive constants $A_i,$ $B_i,$ $C_i,$  $i=1,2,$ the following polynomial growth conditions are satisfied for large $s$:
\begin{equation}\label{growth_cond}
\begin{aligned}
&A_1 s^{\alpha} \leq \gamma_1(s) \leq A_2 s^{\alpha},& &
&C_1 s^{\gamma} \leq \gamma_2(s) \leq C_2 s^{\gamma},& &
&B_1 s^{\beta} \leq \gamma_3(s) \leq B_2 s^{\beta}. 
\end{aligned}
\end{equation}
\qed

In \cite{LuMiNeRe} we have studied the existence and uniqueness of global weak and classical solutions to the Peterlin viscoelastic model  with $\lambda=\gamma_{\tM}=1.$ We have shown the existence of global-in-time weak solutions in both two and three space dimensions with only $\Ccko \in \Lp{p} \cap \LLko{1+\delta}{W^{1,1+\delta}}$ for $p>2$ and $0 < \delta << 1$. Moreover, for the two-dimensional case we have proven the existence and uniqueness of classical solutions to model (MP), which is of interest for our further needs; see Theorem~3 in \cite{LuMiNeRe}.

\begin{theo}(unique classical solution to (MP))\label{theo:uniq1}\\
Let the assumptions \eqref{growth_cond} on $\go,$ $\gt$ and $\gamma_3$ be satisfied with 
\begin{equation}\label{exponents_1}
 \alpha+\beta+1 > 2, \, \alpha>\beta+1,\, \beta \geq 0, \ \mbox{and} \ \gamma  < \alpha+1 \ \mbox{or} \ \gamma =\alpha+1 \ \mbox{with} \ dB_2C_2 < A_1B_1.
\end{equation}
In addition, let $|\psi'(s)|\le Ks^{\beta-1}$ for large $s$. Then there exists a global classical solution to the Peterlin model (MP) for $d=2$.
\end{theo}

In \cite{mizerova} we studied a particular case of (MP) in which $\gamma_3(s)=s,$  and the two functions $\go,$ $\gt$ were taken as in \eqref{growth_cond}. For the two-dimensional case we showed the existence and uniqueness of regular global-in-time weak solutions as defined in \eqref{weak_MP}. Another technique of the proof allowed us to cover different choices of the constitutive functions than in Theorem~\ref{theo:uniq1} above.  For completness, we recall Theorem~5.3 from \cite{mizerova}.

\begin{theo}(unique regular  weak solution to (MP))\label{theo:uniq2}\\
Let $\Omega \subset \mathbb R^2$ be of class $C^2$ and the initial data $(\ucko_0,\Ccko_0) \in [H^2(\Omega)^2 \cap V] \times H^2(\Omega)^{2\times 2}.$ Let the assumptions \eqref{growth_cond} and one of the following conditions be satisfied
\begin{equation}\label{exponents_2}
\begin{aligned}
 0 < \alpha \leq 2,\ 1 \leq \gamma < \alpha+1, \ \mbox{or } \ \gamma = \alpha+1  \ \mbox{with } \ dB_2C_2 < A_1B_1, \\
 \mbox{or } \ \alpha=0 \ \mbox{ and } \ \gamma=1.
 \end{aligned}
\end{equation}
Then the weak solution to the Peterlin model (MP) with $\gamma_3(s)=s$ satisfies 
\begin{align*}
\ucko \in \LLko{\infty}{H^2(\Omega)^2}, \quad \Ccko \in \LLko{\infty}{H^2(\Omega)^{2\times 2}}
\end{align*}
and it is unique.
\end{theo}
The proof of higher regularity is based on the regularity results for the Stokes and the Laplace operators. Analogously as in \cite{Te,LuMiNeRe,mizerova}, assuming enough regular data, conditions \eqref{growth_cond} and \eqref{exponents_2}, we can repeat the  argument several times to obtain arbitrarily regular solution to \eqref{weak_MP}. For our further needs it is sufficient to have $\ucko \in \LLko{\infty}{H^3(\Omega)^2}$ and $\Ccko \in \LLko{\infty}{H^2(\Omega)^{2\times 2}}.$
 
\begin{coro} \ \label{coro:uniq2} \\
Let $\Omega$ be of class $C^3$ and the initial data $(\ucko_0,\Ccko_0) \in [H^3(\Omega)^2 \cap V] \times H^3(\Omega)^{2\times 2}.$ Let the assumptions \eqref{growth_cond} and \eqref{exponents_2} be satisfied. Then the weak solution to the Peterlin model (MP) with $\gamma_3(s)=s$ satisfies 
\begin{align*}
\ucko \in \LLko{\infty}{H^3(\Omega)^2}, \quad \Ccko \in \LLko{\infty}{H^3(\Omega)^{2\times 2}}.
\end{align*}
\end{coro}

Let us conclude the above two results for further reference.
Let the assumptions of Theorem~\ref{theo:uniq1} and Corollary~\ref{coro:uniq2} be satisfied. Then there exists a unique regular weak solution to the Peterlin model (MP), i.e., a couple $(\ucko,\Ccko)$ satisfying \eqref{weak_MP_ucko} - \eqref{weak_MP_Ccko} such that 
\begin{equation}
\begin{aligned}\label{reg_weak_MP}
\ucko &\in  \LLko{2}{V} \cap \LLko{\infty}{H^3(\Omega)^2},\\
 \Ccko &\in \LLko{2}{H^1(\Omega)^{2\times 2}}\cap \LLko{\infty}{H^2(\Omega)^{2\times 2}}.
\end{aligned}
\end{equation}


\section{The Fokker-Planck equation}\label{sec:FP}

In this section we want to prove the existence of the weak solution $\psi=\psis=M\psish$ to the Fokker-Planck equation \eqref{FP_M} for a given couple $(\us,\Cs).$ 

We set $(\ucko,\Ccko)=(\us,\Cs),$ where
\begin{align}\label{us_Cs}
 \us \in \LLko{2}{V} \cap \LLko{\infty}{H^3(\Omega)^2}, \quad \Cs \in \LLko{\infty}{H^2(\Omega)^{d \times d}}
\end{align}
and we seek the solution 
$\psish(t,\xko,\Rko)=\psi(t,\xko,\Rko)/M(\Rko)$ such that 
\begin{subequations}\label{FP_problem}
\begin{align}\label{FP_given_uC}
\pd{\psish}{t} + (\us\cdot\nabla_x)\psish +\diR\left[\nabla_x\us \cdot \Rko \psish\right]&= \Gamma(\trCs)\, \nablaR \cdot \(M \, \nablaR\(\frac{\psish}{M}\)\)+\varepsilon\Delta_x\psish 
\end{align}
subject to the following decay/boundary and initial conditions
\begin{align}
&\left| M \bigg( \Gamma(\trCs)\nablaR\psish-(\nablax\us)\Rko\psish\bigg)\right| \rightarrow 0 \quad \mbox{as} \quad |\Rko| \mapsto \infty && \quad \mbox{ on } (0,T] \times\Omega, \label{FP_decay}\\
&\varepsilon\pd{\psish}{\nko}=0  && \quad \mbox{ on }  (0,T) \times\partial\Omega \times \Rd, \\
&\psish(0)=\psih_0 &&\quad\mbox{ on }  \Omega \times \Rd .
\end{align}
Further we assume that 
\begin{align}\label{FP_ic}
&\psih_0 \in \LpM{2} \ \mbox{ with }  \ \psih_0 \geq 0 \mbox{ a.e. on } \Omega \times \Rd, \quad \intR{M(\Rko)\psih_0(x,\Rko)}=1 \mbox{ a.e. } x \in \Omega.
\end{align}
\end{subequations}
\begin{defin}\label{def:FP_given_uC}
We refer to the system of equations and conditions \eqref{us_Cs}, \eqref{FP_problem} as problem \textnormal{(FP)}.
\end{defin}

Let us point out that the difference between equation \eqref{FP_given_uC} and the Fokker-Planck equation of the Hookean dumbbell model studied in \cite{BaSu6} is the coefficient $\Gamma(\trCs).$ Under the assumptions \eqref{Gamma_vareps}, \eqref{growth_cond} and \eqref{us_Cs} it holds that $\n{\LLko{\infty}{L^{\infty}(\Omega)^{d\times d}}}{\Gamma(\trCs)} \leq c.$ Thus, the whole proof  of the existence result for the Fokker-Planck equation from Section~4 in \cite{BaSu6} can be repeated for problem (FP)  defined above. In what follows we recall the key steps of the proof to recall its main idea.

\subsection{Semi-discrete approximation of a regularized problem}

Firstly, we consider a regularization  ($\text{FP}_{\text{L}}$) of  problem (FP) that is based on the parameter $L>1.$ The term involving $\nablax\us$ in \eqref{FP_given_uC} and the corresponding term in \eqref{FP_decay} are modified using the cut-off function $\betaL \in C(\mathbb R)$ defined as 
\begin{align}
\betaL(s)=\min{(s,L)}=\left\{\begin{array}{ll}
s, & s \leq L \\
L, & s \geq L. \\
\end{array}\right. 
\end{align}
We seek a solution $\psishL \in \LLko{\infty}{\LpM{2}} \cap \LLko{2}{\hat{X}}$  that,  for any $\varphih \in W^{1,1}(0,T;\hat{X})$  with $\varphih(T,\cdot,\cdot)=0,$  satisfies 
\begin{equation}\label{FPL_weak}
\begin{aligned} 
& -\intOTOR{M \psishL \pd{\varphih}{t}} + \intOTOR{M \bigg[ \varepsilon  \nablax\psishL-\us\psishL\bigg] \cdot \nablax \varphih }  + \\
& + \intOTOR{M\bigg[\Gamma(\trCs)\,\nablaR\psishL - \big[(\nablax\us)\Rko\big]\betaL(\psishL)\bigg] \cdot \nablaR\varphih} = & \\
&  = \intOR{M \betaL(\psih_0)\varphih}  . 
\end{aligned}
\end{equation}

In order to prove the existence of weak solutions to ($\text{FP}_{\text{L}}$)  we study a discrete-in-time approximation of \eqref{FPL_weak}. To this end we consider a regular mesh $\{t_1,\ldots,t_N\}$ on the time interval $[0,T].$ For $n=1,\ldots,N,$ we seek the values $\psishL^n \in \hat{Z}_2 \cap \hat{X}$  representing the approximate solution. By  $\psishL^{\Delta t}(t,\cdot)$ we denote the piecewise linear approximation of $\psishL(t,\cdot).$ Further,  we employ a collective symbol $\psishL^{\Delta t(,\pm)}$  for $\psishL^{\Delta t}$ and the values $\psishL^{n},$ $\psishL^{n-1}$ at the end points of the interval $[t_{n-1},t_n].$ For more details  we refer to the works of Barret and S\"uli \cite{BaSu,BaSu2,BaSu6}.

 Now, we recall the most important uniform estimates. As first, it can be shown, cf., Lemma~4.3 in \cite{BaSu6}, that an arbitrary $r$-th moment of the approximate solution is uniformly bounded. 

\begin{lem}(uniform bounds on the moments)\label{lem:moments} \\
Let the assumptions \eqref{us_Cs} and \eqref{FP_ic} be satisfied. Then we have, for any $r \in \mathbb R^+_0,$ that 
\begin{align*}
\intOR{M|\Rko|^r \psishL^n} \leq c, \quad n=0,\ldots,N.
\end{align*}
\end{lem}

The next estimate implies that the solution $\psishL^{\Delta t}$ has finite Fisher information and finite relative entropy with respect to the Maxwellian $M.$ We refer to Lemma~4.4 in \cite{BaSu6}.

\begin{lem}(finite Fisher information and relative entropy)\label{lem:Fisher} \\
Under the assumptions of Lemma~\ref{lem:moments} it holds that 
\begin{equation}
\begin{aligned}
&\esssup_{t\in[0,T]} \intOR{M\mathcal{F}(\psishL^{\Delta t(,\pm)})(t)} + \frac{1}{\Delta t L}\intOR{M(\psishL^{\Delta t,+}-\psishL^{\Delta t,-})^2} + & \label{rel_en_est}\\
& + \intOTOR{M\Bigg[ \left|\nablax\sqrt{\psishL^{\Delta t(,\pm)}}\right|^2 + \left|\nablaR\sqrt{\psishL^{\Delta t(,\pm)}}\right|^2\Bigg]}\leq c.
\end{aligned}
\end{equation}
Moreover, we have that 
\begin{align*}
\left|\intOTOR{M \pd{\psishL^{\Delta t}}{t} \varphih }\right| \leq c\|\varphih\|_{\LLko{2}{W^{1,\infty}(\Omega\times\Rd)}} \quad \forall \varphih \in \LLko{2}{W^{1,\infty}(\Omega\times\Rd)}.
\end{align*}
\end{lem}

The function  $\mathcal{F} \in C(\mathbb R^+)$ appearing in \eqref{rel_en_est} is given by $\mathcal{F}(s)\defeq s(\log s - 1)+1. $ As pointed out in \cite{BaSu6} it is a non-negative, strictly convex function that can be considered to be defined on $[0,\infty]$ with $\mathcal{F}(1)=0.$

\subsection{Existence of weak solutions to \textnormal{(FP)} }

Passage to the limit with $L \rightarrow \infty$ implies the existence of weak solution to (FP) as shown in Theorem~4.1 in \cite{BaSu6}.

\begin{theo}(existence of weak solution to \textnormal{(FP)})\label{theo:weak_FP}\\
Let the assumptions  \eqref{us_Cs}, \eqref{FP_ic} be satisfied, and let $\Delta t \leq (4L^2)^{-1}$ as $L \rightarrow \infty.$ Then, there exists a subsequence of $\{\psishL^{\Delta t}\}_{L>1}$, and a function $\psish$ such that 
\begin{align*}
|\Rko|^r\psish \in \LLko{\infty}{\LpM{1}}, \quad \mbox{ for any } r \in [0,\infty), \\
\psish \in H^1(0,T;M^{-1}[\HkM{s}]'), \quad \mbox{ for any } s > d+1, 
\end{align*}
with \begin{align*}
\psish \geq 0 \ \mbox{a.e. on } [0,T]\times \Omega \times \Rd \ \mbox{ and } \ \intR{M(\Rko)\psish(t,x,\Rko)} \leq 1 \ \mbox{for a.e. } (x,t) \in [0,T]\times \Omega,
\end{align*}
and finite relative entropy and Fisher information, with 
\begin{align*}
\mathcal{F}(\psish) \in \LLko{\infty}{\LpM{1}} \quad \mbox{ and } \quad \sqrt{\psish} \in \LLko{2}{\HkM{1}},
\end{align*}
such that as $L \rightarrow \infty$ (and $\Delta t \rightarrow 0$) 
\begin{align*}
M^{1/2}\nablax \sqrt{\psishL^{\Delta t(,\pm)}} \rightarrow M^{1/2}\nablax\sqrt{\psish}  & & \mbox{ weakly in } \LLko{2}{L^2(\Omega\times\Rd)},\\
M^{1/2}\nablaR \sqrt{\psishL^{\Delta t(,\pm)}} \rightarrow M^{1/2}\nablaR\sqrt{\psish}  & & \mbox{ weakly in } \LLko{2}{L^2(\Omega\times\Rd)},\\
M\pd{\psishL^{\Delta t}}{t}\rightarrow M \pd{\psish}{t} & & \mbox{ weakly in } \LLko{2}{[H^s(\Omega\times\Rd)]'},\\
|\Rko|^r \beta^{L}(\psishL^{\Delta t(,\pm)}), \, |\Rko|^r\psishL^{\Delta t(,\pm)} \rightarrow |\Rko|^r \psish  & & \mbox{ strongly in } \LLko{p}{\LpM{1}},
\end{align*}
for any $p \in [1,\infty).$

Additionally, for $s > d+1,$ the function $\psish$ satisfies 
\begin{equation}\label{FP_weak}
\begin{aligned} 
& -\intOTOR{M \psish \pd{\varphih}{t}} + \intOTOR{M \bigg[ \varepsilon  \nablax\psish-\us\psish\bigg] \cdot \nablax \varphih }  + \\
& + \intOTOR{M\bigg[\Gamma(\trCs)\,\nablaR\psish - \big[(\nablax\us)\Rko\big]\psish\bigg] \cdot \nablaR\varphih} = & \\
&  = \intOR{M \psih_0(x,\Rko)\varphih(0,x,\Rko)} , \quad \forall \varphih \in W^{1,1}(0,T;\hat{X}) \mbox{ with } \varphi(T,\cdot,\cdot)=0.
\end{aligned}
\end{equation}
\end{theo}

\subsection{Macroscopic closure}

In what follows we shall discuss the rigorous macroscopic closure of the Fokker-Planck equation. 
It can be shown that 
under the assumptions of Theorem~\ref{theo:weak_FP} the weak solution $\psish$ to problem (FP) is such that 
\begin{align*}
\Ccko(M\psish)=\Ccko(M\psish)^T \in \LLko{\infty}{L^2(\Omega)^{d\times d}} \cap\LLko{2}{H^1(\Omega)^{d \times d}},
\end{align*}
and it satisfies, for any $\Dcko \in W^{1,1}(0,T; H^1(\Omega)^{d \times d}),$
\begin{equation}\label{weak_closure}
\begin{aligned}
& -\intoto{T}{\Ccko(M\psish) : \pd{\Dcko}{t}} + \intoto{T}{\varepsilon  \nablax\Ccko(M\psish):\nablax\Dcko-(\us\cdot\nablax)\Dcko:\Ccko(M\psish)  } -& \\
& -\intoto{T}{(\nablax\us)\Ccko(M\psish)+\Ccko(M\psish)(\nablax\us)^T}-  \intoto{T}{2\Gamma(\trCs)\bigg[\Icko -\Ccko(M\psish)\bigg]:\Dcko} = & \\
&  = \intOR{\Ccko(M\psih_0)(x):\Dcko(0,x)} .
\end{aligned}
\end{equation}
Let us note that by \eqref{m2_Mcond} and \eqref{Gamma_vareps} it holds that 
\begin{align*}
 \intoto{T}{2\Gamma(\trCs)\bigg[\Icko -\Ccko(M\psish)\bigg]:\Dcko}= \intoto{T}{\left[\frac{\gt(\trCs)}{\lambda}\Icko -\frac{\go(\trCs)}{\lambda\gamma_{\tM}}\Ccko(M\psish)\right]:\Dcko}.
\end{align*}

For the careful derivation we refer to Lemmas~4.2, 4.5 and 4.6 in \cite{BaSu6}. The main idea is to test the semi-discrete approximation of the Fokker-Planck equation with $\Rko\otimes\Rko:\Dcko \in \hat{X}$ for $\Dcko\in C^{\infty}(\bar{\Omega}).$  The definition of the conformation tensor along with some useful identities mentioned below yields all the terms in \eqref{weak_closure} except the term containing the gradient $\nablaR\psishL^n.$  In the latter term we have to integrate by parts with respect to $\Rko,$ which requires the approximation of $\psishL^n$ by a sequence of smooth functions. The dense embedding of $C^{\infty}(\bar{\Omega};C_0^{\infty}(\Rd))$ in $\hat{X},$ cf., \cite{BaSu2}, then implies that the closure is indeed valid for $\psishL^n \in \hat{X}$ itself. 

Here we only present  \textit{formal} macroscopic closure of the weak formulation \eqref{FP_weak}. To this end let $\varphih$ in \eqref{FP_weak} be $\Rko\otimes\Rko:\Dcko$  with $\Dcko\in W^{1,1}(0,T;C^{\infty}(\bar{\Omega}))$ such that $\Dcko(T,\cdot)=0.$ Taking into account the definition of the conformation tensor $\displaystyle \Ccko(\psi)\defeq \langle\Rko\otimes\Rko\rangle=\intR{\Rko\otimes\Rko \,\psi(t,x,\Rko)}$ we directly get  
\begin{equation*}
\begin{aligned} 
& -\intoto{T}{\Ccko(M\psish) : \pd{\Dcko}{t}} + \intoto{T}{\varepsilon  \nablax\Ccko(M\psish):\nablax\Dcko-(\us\cdot\nablax)\Dcko:\Ccko(M\psish)  } + \\
& + \intOTOR{M\bigg[\Gamma(\trCs)\,\nablaR\psish - \big[(\nablax\us)\Rko\big]\psish\bigg] \cdot \nablaR(\Rko\otimes\Rko:\Dcko)} = & \\
&  = \intOR{\Ccko(M\psih_0)(x):\Dcko(0,x)}.
\end{aligned}
\end{equation*}

Moreover, for any $a\in \Rd$ it holds that $(a\cdot\nablaR)(\Rko\otimes\Rko)=a\Rko^T+\Rko a^T.$ Thus, the latter identity with  $a=M\big[(\nablax\us)\Rko\big]\psish$ yields
\begin{align*}
&\intOTOR{M\big[(\nablax\us)\Rko\big]\psish \cdot \nablaR(\Rko\otimes\Rko:\Dcko)}& \\
&\hskip 5cm  =\intoto{T}{(\nablax\us)\Ccko(M\psish)+\Ccko(M\psish)(\nablax\us)^T}.
\end{align*}
Further, \textit{formal} integration by parts, which is rigorously done by employing the density argument mentioned above, yields the  term 
\begin{equation*}
\begin{aligned} 
&\intOTOR{M\Gamma(\trCs)\,\nablaR\psish  \cdot \nablaR(\Rko\otimes\Rko:\Dcko)} = & \\
&\hskip 5cm   = - \intOTOR{ \Gamma(\trCs)\,\psish \bigg[\nablaR \cdot \big(M\nablaR(\Rko\otimes\Rko)\big)\bigg]:\Dcko}.
\end{aligned}
\end{equation*}
By \eqref{Maxwellian} and identity $\DeltaR (\Rko\otimes\Rko)=2\Icko,$ we finally obtain
\begin{equation*}
\begin{aligned} 
&- \intOTOR{ \Gamma(\trCs)\,\psish \bigg[\nablaR \cdot \big(M\nablaR(\Rko\otimes\Rko)\big)\bigg]:\Dcko}=& \\
&\hskip 5cm =-  \intoto{T}{2\Gamma(\trCs)\bigg[\Icko -\Ccko(M\psish)\bigg]:\Dcko},
\end{aligned}
\end{equation*}
and thus equation \eqref{weak_closure}.


\section{The existence result for the kinetic Peterlin model}\label{sec:ex_kinetic}

In what follows we combine the results from the previous two sections to prove the existence of large-data global-in-time weak solutions to the kinetic Peterlin model (KP). Let us note that the weak solutions to the Fokker-Planck equation for a given pair $(\us,\Cs)$ exist in both two and three space dimensions, cf., Section~\ref{sec:FP}. However, the main result is only valid in two space dimensions due to the uniqueness result for the macroscopic model from Section~\ref{sec:uniq_macro}.

\begin{theo}(existence of weak solution to \textnormal{(KP)})\label{theo:weak_KP}\\
Let $d=2$ and $\Omega$ be of class $C^3.$ Let $\ucko_0 \in V \cap H^3(\Omega)^2$ and let $\psih_0$ satisfy \eqref{FP_ic} with $\Ccko_0:=\Ccko(M\psih_0) \in H^3(\Omega)^{2\times 2}.$ It follows that there exists a couple $(\ucko_{\tP},\Ccko_{\tP})$ satisfying \eqref{reg_weak_MP} and solving the weak formulation  \eqref{weak_MP} of the Peterlin model. 

In addition, there exists $\psih_{\tP}$ satisfying 
\begin{subequations}\label{psis_solution}
\begin{align}
|\Rko|^r\psih_{\tP} \in \LLko{\infty}{\LpM{1}}, \quad \mbox{ for any } r \in [0,\infty), 
\end{align}
with
\begin{align}
\psih_{\tP} \geq 0 \ \mbox{ a.e. on} [0,T] \times \Omega \times \Rd \ \mbox{and} \ \intR{M(\Rko)\psih_P(t,\xko,\Rko)} =1 \ \mbox{ for a.e.} (t,\xko) \in [0,T]\times\Omega,\\
 \mathcal{F}(\psih_P) \in \LLko{\infty}{\LpM{1}} \ \mbox{ and } \ \sqrt{\psih_P} \in \LLko{2}{\HkM{1}},
\end{align}
and solving 
\begin{align}
& -\intOTOR{M \psih_{\tP} \pd{\varphih}{t}} + \intOTOR{M \bigg[ \varepsilon  \nablax\psih_{\tP}-\ucko_{\tP}\psih_{\tP}\bigg] \cdot \nablax \varphih }  + \\
& + \intOTOR{M\bigg[\Gamma(\trC_{\tP})\,\nablaR\psih_{\tP} - \big[(\nablax\ucko_{\tP})\Rko\big]\psih_{\tP}\bigg] \cdot \nablaR\varphih} = \\
& = \intOR{M \psih_0(x,\Rko)\varphih(0,x,\Rko)},\quad  \forall \varphih \in W^{1,1}(0,T;\hat{X}) \ \mbox{ with } \ \varphi(T,\cdot,\cdot)=0. 
\end{align}
Moreover, we have that $\Ccko_{\tP}=\Ccko(M\psih_{\tP}).$ 
\end{subequations}
\end{theo}
\begin{proof}
The existence of $(\ucko_{\tP},\Ccko_{\tP})$ satisfying \eqref{us_Cs} as a unique solution to \eqref{weak_MP} is a straightforward consequence of Theorem~\ref{theo:uniq1} and Corollary~\ref{coro:uniq2}.  Theorem~\ref{theo:weak_FP} with $(\us,\Cs)=(\ucko_{\tP},\Ccko_{\tP})$ yields the existence of $\psih_{\tP}$ satisfying \eqref{psis_solution}. Comparing \eqref{weak_MP} with \eqref{weak_closure} and recalling uniqueness of the regular weak solution to (MP) from Section~\ref{sec:uniq_macro}, we can conclude that $\Ccko_{\tP}=\Ccko(M\psih_{\tP}).$ 
\end{proof}

 \section*{Acknowledgements}
This work was partially supported by the Simons - Foundation grant 346300, the Polish Government MNiSW 2015-2019 matching fund, and the German research foundation (DFG) grant TRR 146 ``Multiscale simulation methods for soft matter systems''. The authors would like to thank Gabriella Puppo (Universit\` a degli Studi dell'Insubria) and Endre S\"uli (University of Oxford) for fruitful discussion on the topic.

\bibliographystyle{siam}
\bibliography{NSFP_references}

\begin{thebibliography}{10}

\bibitem{BaSu}
{\sc J.~W. Barrett and E.~S{\"u}li}, {\em {Existence and equilibration of
  global weak solutions to kinetic models for dilute polymers {I}: finitely
  extensible nonlinear bead-spring chains}}, Math. Models Methods Appl. Sci.,
  21 (2011), pp.~1211--1289.

\bibitem{BaSu2}
\leavevmode\vrule height 2pt depth -1.6pt width 23pt, {\em {Existence and
  equilibration of global weak solutions to kinetic models for dilute polymers
  {II}: {H}ookean-type models}}, Math. Models Methods Appl. Sci., 22 (2012).

\bibitem{BaSu5}
\leavevmode\vrule height 2pt depth -1.6pt width 23pt, {\em Existence of global
  weak solutions to compressible isentropic finitely extensible bead-spring
  chain models for dilute polymers}, Math. Models Methods Appl. Sci., 26
  (2016), pp.~469--568.

\bibitem{BaSu4}
\leavevmode\vrule height 2pt depth -1.6pt width 23pt, {\em Existence of global
  weak solutions to compressible isentropic finitely extensible nonlinear
  bead-spring chain models for dilute polymers: the two-dimensional case}, J.
  Differential Equations, 261 (2016), pp.~592--626.

\bibitem{BaSu6}
\leavevmode\vrule height 2pt depth -1.6pt width 23pt, {\em {Existence of global
  weak solutions to the kinetic Hookean dumbbell model for incompressible
  dilute polymeric fluids}}, ArXiv,  (2017).

\bibitem{degond_liu}
{\sc P.~Degond and H.~Liu}, {\em {Kinetic models for polymers with inertial
  effects}}, Netw. Heterog. Media, 4 (2009), pp.~625--647.

\bibitem{Dub}
{\sc J.~A. Dubinski\u{i}}, {\em {Weak convergence for nonlinear elliptic and
  parabolic equations}}, Mat. Sb. (N.S.), 67 (1965), pp.~609--642.

\bibitem{LuMiNeRe}
{\sc M.~{Luk{\'a}\v{c}ov{\'a}-Medvi{\v d}ov{\'a}}, H.~Mizerov{\'a},
  {Ne\v{c}asov{\'a}, \v{S}}., and M.~Renardy}, {\em {Global existence result
  for the {P}eterlin viscoelastic model}}, SIAM J. Math. Anal.,  (2017).

\bibitem{LuMiNoTaI}
{\sc M.~{Luk\'{a}\v{c}ov\'{a}-Medvi\v{d}ov\'{a}}, H.~Mizerov\'{a}, H.~Notsu,
  and M.~Tabata}, {\em {Numerical analysis of the {Oseen}-type {Peterlin}
  viscoelastic model by the stabilized {L}agrange--{G}alerkin method,
  {P}art~{I}: A nonlinear scheme}}, ESAIM:~M2AN, in press,  (2017).

\bibitem{LuMiNoTaII}
\leavevmode\vrule height 2pt depth -1.6pt width 23pt, {\em {Numerical analysis
  of the {Oseen}-type {Peterlin} viscoelastic model by the stabilized
  {L}agrange--{G}alerkin method, {P}art~{II}: A linear scheme}}, ESAIM:~M2AN,
  in press,  (2017).

\bibitem{mizerova}
{\sc H.~Mizerov{\'a}}, {\em {Analysis and simulation of some viscoelastic
  fluids}}, {PhD thesis (in preparation)}, University of Mainz, Germany, 2015.

\bibitem{P}
{\sc A.~Peterlin}, {\em {Hydrodynamics of macromolecules in a velocity field
  with longitudinal gradient}}, J. Polymer Sci. Part. B., Polymer Lett.,
  (1966), pp.~287--291.

\bibitem{Re1}
{\sc M.~Renardy}, {\em {Mathematical analysis of viscoelastic flows}},
  {CBMS-NSF Conference Series in Applied Mathematics 73}, Society for
  Industrial and Applied Mathematics, 2000.

\bibitem{renardy_wang}
{\sc M.~Renardy and T.~Wang}, {\em {Large amplitude oscillatory shear flows for
  a model of a thixotropic yield stress fluid}}, J. Non-Newton. Fluid, 222
  (2015), pp.~1--17.

\bibitem{shieber}
{\sc J.~D. Shieber}, {\em {Generalized {B}rownian configuration field for
  {F}okker--{P}lanck equations including center-of-mass diffusion}}, J.
  Non-Newton. Fluid, 135 (2006), pp.~179--181.

\bibitem{Te}
{\sc R.~Temam}, {\em {Navier-Stokes Equations: Theory and Numerical Analysis}},
  North-Holland Publishing Company, Amsterdam, New York, Oxford, 1977.

\end{thebibliography}
\end{document}